\documentclass[11pt]{amsart}
\usepackage{geometry}                
\usepackage[parfill]{parskip}    
\usepackage{graphicx}
\usepackage{color}
\usepackage{amsmath,amscd}
\usepackage{amssymb}
\usepackage{epstopdf}
\usepackage[all]{xy}
\usepackage{eucal}
\usepackage{euler}
\usepackage{times}
\title{Amenable actions, invariant means and bounded cohomology}

\DeclareMathOperator{\supp}{Supp}

\def\note#1{{ \color{red}#1}}

\theoremstyle{plain}
\newtheorem{thm}{Theorem}
\newtheorem*{mainthm}{Theorem A}
\newtheorem*{mainthmtwo}{Theorem B}
\newtheorem*{cor}{Corollary}
\newtheorem{lemma}[thm]{Lemma}

\newtheorem{defn}[thm]{Definition}

\theoremstyle{definition}

\newtheorem{example}[thm]{Example}

\newtheorem{remark}[thm]{Remark}

\newcommand{\eps}{\varepsilon}

\newcommand{\norm}[1]{{\|{#1}\|}}

\newcommand{\ra}{\rightarrow}
\newcommand{\R}{\mathbb{R}}
\newcommand{\N}{\mathbb{N}}
\newcommand{\woo}{W_{00}}
\newcommand{\wo}{W_{0}}
\newcommand{\no}{N_0}
\newcommand{\noo}{N_{00}}
\newcommand{\cale}{\mathcal{E}}
\newcommand{\xr}{\xrightarrow}
\newcommand{\cx}{C(X)}
\newcommand{\gcx}{\text{$G$-$C(X)$}}
\let\bullet =\null
\let\note=\null
\thanks{JB, GN and NW were partially supported by EPSRC grant  EP/F031947/1. 
PN was partially supported by NSF grant DMS-0900874}

\author{Jacek Brodzki}
\address{School of Mathematics, University of Southampton, Highfield, Southampton, SO17 1SH, England}
\email{J.Brodzki@soton.ac.uk}

\author{Graham A. Niblo}
\address{School of Mathematics, University of Southampton, Highfield, Southampton, SO17 1SH, England}
\email{G.A.Niblo@soton.ac.uk}

\author{Piotr W. Nowak}
\address{Department of Mathematics, Texas A\& M University, College Station, TX 77840}
\email{pnowak@math.tamu.edu}

\author{Nick Wright}
\address{School of Mathematics, University of Southampton, Highfield, Southampton, SO17 1SH, England}
\email{N.J.Wright@soton.ac.uk}

\begin{document}
\begin{abstract}
We show that  topological amenability of an action of a countable discrete group  on a compact space
is equivalent to  the existence of an invariant 
mean for the action. We prove also  that this is equivalent to vanishing of
bounded cohomology for a class of Banach G-modules associated to the action, as well as to vanishing of a specific cohomology class. 
In the case when the compact space is a point our result reduces to a classic theorem 
of B.E.~Johnson characterising amenability of groups. In the case when the compact
space is the Stone-\v{C}ech compactification of the group we obtain a cohomological characterisation 
of exactness for the group, answering a question of Higson.
\end{abstract}

\maketitle

\section{Introduction}


An invariant mean on a countable discrete group $G$ is a positive linear functional on $\ell^\infty(G)$ which is normalised by the requirement that it pairs with the constant function 1 to give 1, and 
which is fixed by the natural action of $G$ on the space $\ell^\infty(G)^*$.
A group is said to be amenable if it admits an invariant mean. The notion of an amenable action of a group on a topological space, studied by Anantharaman-Delaroche and Renault \cite{AR}, generalises the concept of amenability, and arises naturally in many areas of mathematics. For example, a group acts amenably on a point if and only if it is amenable, while every hyperbolic group acts amenably on its Gromov boundary. 

In this paper we introduce the notion of an invariant mean  for a topological action and prove that the existence of such a mean characterises amenability of the action. 
Moreover, we use the existence of the mean to prove vanishing of bounded cohomology of $G$ with coefficients in a suitable class of Banach $G$ modules, and conversely we prove that 
vanishing of these cohomology groups characterises amenability of the action. This generalises the 
results of 
Johnson \cite{Johnson} on bounded cohomology for amenable groups.


Another generalisation of amenability, this time for metric spaces, was given by Yu \cite{Yu} with the definition of property A. Higson and Roe \cite{HR}  proved a remarkable result that unifies the two approaches: A finitely generated discrete  group $G$ (regarded as a metric space) has Yu's property A if and only if the action of $G$ on its Stone-\v{C}ech compactification $\beta G$ is topologically amenable, and this is true if and only if $G$ acts amenably on any compact space. Ozawa  proved \cite{ozawa} that such groups are exact, and indeed property A and exactness are equivalent for countable discrete groups equipped with a proper left-invariant metric.



To generalise the concept of invariant mean to the context of a topological action, we introduce a Banach $G$-module $W_0(G,X)$ which is an analogue of $\ell^1(G)$, encoding both the group and the space on which it acts. Taking the dual and double dual of this space we obtain analogues of $\ell^\infty(G)$ and $\ell^\infty(G)^*$.  
A mean for the action is an element $\mu\in \wo(G,X)^{**}$ satisfying the normalisation condition $\mu(\pi)=1$, where the element $\pi$   is a summation operator, corresponding to the pairing of $\ell^1(G)$ with the constant function 1 in $\ell^\infty(G)$. 
A mean $\mu$ is said to be invariant if $\mu(g\cdot \varphi)=\mu(\varphi)$
for  every $\varphi\in W_{0}(G,X)^*$, (Definition \ref{mean}).

With these
notions in place we give the following very natural characterisation of amenable actions.

\begin{mainthm}
Let $G$ be a countable discrete group acting by homeomorphisms on a compact Hausdorff topological space $X$. The action is amenable
if and only if there exists an invariant mean for the action.
\end{mainthm}

We then turn to the question of a cohomological characterisation of amenable actions. 
Given an action of a countable discrete group $G$ on a compact space $X$ by homeomorphisms we
introduce a submodule $\no(G,X)$ of $\wo(G,X)$ associated to the action and which is analogous to the submodule $\ell^1_0(G)$ of $\ell^1(G)$ consisting of all functions of sum $0$. Indeed when $X$ is a point these modules coincide. We also define
a cohomology class $[J]$, called the Johnson class of the action, which lives in the first bounded
cohomology group of $G$ with coefficients in the module $\no(G,X)^{**}$. 
We have the following theorem.

\begin{mainthmtwo}
Let $G$ be a countable discrete group acting by homeomorphisms on a compact Hausdorff topological space $X$. Then the following are equivalent
\begin{enumerate}
\item \label{amen} The action of $G$ on $X$ is topologically amenable.
\item \label{J} The class $[J]\in H_b^1(G, N_0(G,X)^{**})$ is trivial.
\item \label{vanish}   $H_b^p(G,\cale^*) = 0 $ for $p\geq 1$ and every $\ell^1$-geometric $\gcx$ module $\cale$. 
\end{enumerate}
\end{mainthmtwo}

The definition of $\ell^1$-geometric $\gcx$ module  is given in Section \ref{section : geometric modules}.
When $X$ is a point our theorem reduces to Johnson's celebrated characterisation of amenability
\cite{Johnson}.
As a corollary we also obtain a cohomological characterisation of exactness for discrete groups, which answers a question of Higson, and which follows 
from our main result when $X$ is the Stone-\v{C}ech compactification $\beta G$ of the group $G$. In this case, $C(\beta G)$ can be identified with 
$\ell^\infty(G)$, and we obtain the following. 
\begin{cor}
Let $G$ be a countable discrete group. Then the following are equivalent. 
\begin{enumerate}
\item The group $G$ is exact; 
\item The Johnson class $[J]\in H^1_b(G,\no(G,\beta G)^{**})$ is trivial; 
\item $H^p_b(G,\cale^*) = 0$ for $p\geq 1$ and every $\ell^1$-geometric $G$-$\ell^\infty(G)$-module $\cale$. 
\end{enumerate}
\end{cor}


This paper builds on the cohomological characterisation of property A developed in \cite{BrodzkiNibloWright} and on the study of cohomological properties of exactness in \cite{DouglasNowak}. 


\section{Geometric Banach modules}

Let $C(X)$ denote the space of real-valued continuous functions on $X$.  For a  function $f: G\ra C(X)$ we shall denote by $f_g$ the continuous function on $X$ obtained by evaluating $f$ at  $g\in G$.  We define  the $\sup-\ell^1$ norm of $f$ to be 
$$\|f\|_{\infty, 1}=\sup\limits_{x\in X}\sum\limits_{g\in G}|f_g(x)|,$$ 
and denote by  $V$ the Banach space of all functions on $G$ with values in $C(X)$ that have finite norm. We introduce a Banach $G$-module associated to the action.

\begin{defn}
Let  $\woo(G,X)$ be the subspace of $V$ consisting of all functions $f: G\ra C(X)$ which have finite support and such that for some $c\in \R$, depending on $f$, 
$\sum_{g\in G} f_g = c 1_X$, where $1_X$ denotes the constant function $1$ on $X$. The closure of this space in the $\sup-\ell^1$-norm will be denoted $\wo(G,X)$. 
\end{defn}

Let $\pi: W_{00}(G,X) \ra \R$ be defined by  $\sum_{g\in G}f_g = \pi(f)1_X $. The map $\pi$ is continuous with respect to the $\sup-\ell^1$ norm and so extends 
to the closure $W_0 (G,X)$; we denote its kernel by $N_0(G,X)$. 

In the case of $X=\beta G$ and $C(\beta G)=\ell^\infty(G)$ the space $W_{0}(G,\beta G)$
was introduced in \cite{DouglasNowak}. 
For every $g\in G$ we define the function $\delta_g\in \woo(G,X)$ by $\delta_g(h) = 1_X$ when $g=h$, and  zero otherwise.

The $G$-action on $X$ gives an isometric action of $G$ on $C(X)$ in the usual way: for $g\in G$ and $f\in C(X)$, we have 
$(g\cdot f)(x) = f(g^{-1}x)$.  The group $G$ also acts isometrically on the space $V$ in a natural way:  for $g,h\in G, f\in V, x\in X$, we have $(gf)_h(x)=f_{g^{-1}h}(g^{-1}x)= (g\cdot f_{g^{-1}h})(x)$.

Since the summation map $\pi$ is $G$-equivariant (we assume that the action of $G$ on $\R$ is trivial) the action of $G$ restricts to $W_{00}(G,X)$ and so by continuity it restricts to $W_0(G,X)$. We obtain a short exact sequence of $G$-vector spaces: 
\[
0 \ra N_0(G,X)\ra \wo(G,X) \xr{\pi} \R \ra 0.
\]

\begin{defn}
Let $\cale$ be a Banach space. We say that $\cale$ is a $C(X)$-module if it is equipped with a contractive unital representation of the Banach algebra $C(X)$. 

If $X$ is a $G$-space then a $C(X)$-module $\cale$ is said to be a \gcx-module if the group $G$ acts on $\cale$ by isometries and the representation of $C(X)$ is $G$-equivariant.

\end{defn}

Note that the fact that we will only ever consider unital representations of $C(X)$ means that there is no confusion between multiplying by a scalar or by the corresponding constant function. For instance, for $f\in W_{0}(G,X)$ multiplication by $\pi(f)$ agrees with multiplication by $\pi(f)1_X$.

\begin{example}
The space $V$ is a \gcx-module. Indeed, for every $f\in V$ and $t\in C(X)$ we define $tf\in V$ by 
$(tf)_g(x) = t(x)f_g(x)$, for all $g\in G$. This action is well-defined as $\norm{tf}_{\infty,1} \leq\norm{t}_\infty\norm{f}_{\infty,1}$; this also implies that the representation of $C(X)$ on $V$ is contractive. As remarked above, the group $G$ acts isometrically on $V$. 
The representation of $C(X)$ is clearly unital and also equivariant, since for every $g\in G$, $f\in V$ and $t\in C(X)$
\[
\begin{split}
(g (tf))_h(x)& = (tf)_{g^{-1}h}(g^{-1}x)= t(g^{-1}x)f_{g^{-1}h}(g^{-1}x) = (g\cdot t)(x)(gf)_h(x)  
\end{split}
\]
Thus we have $g (tf) = (g\cdot t)(g f)$. 

The equivariance of the summation map $\pi$ implies that both $\wo(G,X)$ and $\no(G,X)$ are $G$-invariant subspaces of $V$. Note however, that $\wo(G,X)$ is not invariant under the action of $C(X)$ defined above, as for 
$f\in \wo(G,X)$ and $t\in C(X)$ we have 
\[
\sum_{g\in G}(tf)_g(x) = \sum_{g\in G}t(x)f_g(x)= t(x)\sum_{g\in G}f_g(x) = ct(x).
\]
 However, the same calculation shows that the subspace $\noo(G,X)$ is invariant under the action of $C(X)$, and so is a \gcx-module,  and hence so is its closure$\no(G,X)$.

%
\end{example}

Let  $\cale$ be a \gcx-module, let $\cale^*$ be the Banach dual of $\cale$ and let $\langle -,-\rangle$ be the pairing between the two spaces. The induced actions of $G$ and $C(X)$ on $\cale^*$ are defined as follows. 
For $\alpha \in \cale^*$, $g\in G$, $f\in C(X)$,  and $v\in \cale$ we let
\[
\langle g\alpha, v\rangle = \langle \alpha ,g^{-1}v\rangle, \qquad 
\langle f \alpha, v\rangle = \langle \alpha, f v\rangle. 
\]
Note that the action of $C(X)$ is well-defined since $C(X)$ is commutative. it is easy to check the following.

\begin{lemma}
If $\cale$ is a \gcx module, then so is $\cale^*$. 
\end{lemma}

We will now introduce a geometric condition on Banach modules which will play the role of an orthogonality condition. To motivate the definition that follows, let us note that if $f_1$ and $f_2$ are functions with disjoint supports on a space $X$ then (assuming that the relevant norms are finite) the $\sup$-norm satisfies the identity $\norm{f_1+f_2}_\infty = \sup \{ \norm{f_1}_\infty, 
\norm{f_2}_\infty\}$, while for the  $\ell^1$-norm we have $\norm{f_1+f_2}_{\ell^1}
 = \norm{f_1}_{\ell^1}+ \norm{f_2}_{\ell^1}$. 
\begin{defn}\label{definition : geometric modules}
Let $\cale$ be a Banach space and a $C(X)$-module. 
We say that $v_1$ and $v_2$ in $\cale$ are disjointly supported if there exist $f_1,f_2\in \cx$ with disjoint supports such that
$f_1v_1 = v_1$ and $f_2v_2 = v_2$. 

We say that the module $\cale$ is $\ell^\infty$-geometric if, whenever $v_1$ and $v_2$ have disjoint supports,  $\norm{v_1+v_2} = sup \{\norm{v_1},\norm{v_2}\}$.

We say that the module $\cale$ is $\ell^1$-geometric if for every two disjointly supported $v_1$ and $v_2$ in 
$\cale$ $\norm{v_1+v_2} = \norm{v_1} + \norm{v_2}$. 
\end{defn}

If $v_1$ and $v_2$ are disjointly supported elements of $\cale$ and $f_1$ and $f_2$ are as in the definition, then $f_1v_2 = f_1f_2v_2 = 0$, and similarly $f_2v_1=0$. 

Note also that  the functions $f_1$ and $f_2$ can be chosen to be of norm one in the supremum norm on $C(X)$.  
To see this, note that Tietze's extension theorem allows one to construct  continuous functions $f_1', f_2'$ on $X$ which are of norm one, have disjoint supports and such that $f_i'$ takes the 
value $1$ on $\supp f_i$ . Then 
$f_i'\phi_i = (f_i'f_i)\phi_i = f_i\phi_i = \phi_i$. Now replace $f_i$ with $f_i'$.

Finally,  if $f_1,f_2\in \cx$ have disjoint supports then, again by Tietze's extension theorem, $f_1v_1$ and $f_2v_2$ are disjointly supported for all 
$v_1,v_2\in \cale$.

\begin{lemma}
If   $\cale$ is an  $\ell^1$-geometric  module then $\cale^*$ is $\ell^\infty$-geometric. 

If  $\cale$ is an $\ell^\infty$-geometric module then  $\cale^*$ is $\ell^1$-geometric. 
\end{lemma}
\begin{proof}
Let us assume that  $\phi_1,\phi_2\in \cale^*$ are disjointly supported and let $f_1,f_2\in C(X)$ be as in Definition \ref{definition : geometric modules}, chosen to be of norm $1$. 

If  $\cale$ is 
$\ell^1$-geometric, then for every vector $v\in \cale$,
$\norm{f_1v} + \norm{f_2v} = \norm{(f_1+f_2)v}\leq \norm{v}$. Furthermore, 
\[
\begin{split}
\norm{\phi_1+\phi_2} & = \sup_{\norm{v}=1}|\langle\phi_1+\phi_2,v\rangle| = \sup_{\norm{v}=1}|\langle f_1\phi_1,v\rangle +\langle f_2\phi_2,v\rangle|\\
& = \sup_{\norm{v}=1}|\langle\phi_1,f_1v\rangle+\langle\phi_2,f_2v\rangle|
\\
& \leq \sup_{\norm{v}=1}(\norm{\phi_1}\norm{f_1v}+\norm{\phi_2}\norm{f_2v})
\\
& \leq \sup\{\norm{\phi_1},\norm{\phi_2}\}
\sup_{\norm{v}=1}(\norm{f_1v} + \norm{f_2v})\\
&
\leq \sup\{\norm{\phi_1},\norm{\phi_2}\}
\end{split}
\]
Since $f_1\phi_2 = 0$ we have that 
\[
\norm{\phi_1} = \norm{f_1(\phi_1+\phi_2)}\leq\norm{f_1} \norm{\phi_1+\phi_2} = \norm{\phi_1+\phi_2}. 
\]
Similarly, we have $\norm{\phi_2} \leq \norm{ \phi_1+\phi_2}$,  and the two estimates together ensure  that $\norm{\phi_1 +\phi_2} =\sup\{\norm{\phi_1},\norm{\phi_2}\}$ as required.

For the second statement, let us assume that $\cale$ is $\ell^\infty$-geometric and that $\phi_1,\phi_2 \in \cale^*$ are 
disjointly supported. Then
\[
\begin{split}
\norm{\phi_1}+\norm{\phi_2} & = \sup_{\norm{v_1},\norm{v_2}=1}
\langle\phi_1,v_1\rangle +\langle\phi_2,v_2\rangle \\
&  =  \sup_{\norm{v_1},\norm{v_2}=1} \langle \phi_1, f_1v_1\rangle +\langle \phi_2, f_2v_2\rangle\\
&  =  \sup_{\norm{v_1},\norm{v_2}=1} \langle \phi_1+\phi_2,  f_1v_1 + f_2v_2\rangle\\
& \leq  \sup_{\norm{v_1},\norm{v_2}=1} \norm{\phi_1+\phi_2}\norm{f_1v_1+f_2v_2}\\
& \leq \norm{\phi_1+\phi_2} \leq \norm{\phi_1} + \norm{\phi_2}. 
\end{split}
\]
where the last inequality is just the triangle inequality, so the inequalities are equalities throughout and $\norm{\phi_1}+\norm{\phi_2}=\norm{\phi_1+\phi_2}$ as required.
\end{proof}

We have already established that $\no(G,X)^{}$ is a \gcx-module. Let $\phi^1$ and 
$\phi^2$ be disjointly supported elements of $\no(G,X)$; this means that there exist disjointly supported  functions $f_1$ and $f_2$ in $C(X)$ such that 
$\phi^i = f_i\phi^i$ for $i=1,2$. Then 
\[
\norm{\phi^1+\phi^2}_{\infty,1} = \norm{f_1\phi^1 + f_2\phi^2} = \sup_{x\in X}\sum_{g\in G}|f_1(x)\phi^1_g(x) + f_2(x)\phi^2_g(x) | 
\]
We note that the two terms on the right are disjointly supported functions on $X$ and 
so 
\[
\norm{\phi^1+\phi^2}_{\infty,1} =  \sup_{x\in X}\left(\sum_{g\in G}|f_1(x)\phi^1_g(x)| + \sum_{g\in G}|f_2(x)\phi^2_g(x)| \right)
= \sup(\norm{\phi^1}_{\infty,1}, \norm{\phi^2}_{\infty,1}).\]
Thus we obtain

\begin{lemma}\label{lemma : N0 is ell-infty geometric}
The module $\no(G,X)$ is $\ell^\infty$-geometric.
Hence the dual $\no(G,X)^*$ is $\ell^1$-geometric and the double dual $\no(G,X)^{**}$ is $\ell^{\infty}$-geometric.
\end{lemma}

We now assume that $\cale$ is an  $\ell^1$-geometric $C(X)$-module, so that its dual $\cale^*$ is  $\ell^\infty$-geometric.

\begin{lemma}
Let $f_1,f_2\in \cx$ be non-negative functions such that $f_1+f_2 \leq 1_X$. Then for every 
$\phi_1,\phi_2\in \cale^*$
\[
\norm{ f_1\phi_1 +f_2\phi_2} \leq \sup\{\norm{\phi_1},\norm{\phi_2}\}.
\]
\end{lemma}
\begin{proof}
Let $M\in \N$ and $\eps=1/M$. For $i=1,2$ define $f_{i,0} = \min\{f_i,\eps\}$, 
$f_{i,1} = \min\{f_i-f_{i,0},\eps\}$, $f_{i,2}=\min\{ f_i-f_{i,0}-f_{i,1}, \eps\}$, and so on, to $f_{i,M-1}$. 

Then $f_{i,j}(x) = 0$ iff $f_i(x) \leq j\eps$, so $f_{i,j} > 0$ iff $f_i(x) > j\eps$ which implies that 
$\supp f_{i,j} \subseteq f^{-1}_i([j\eps,\infty))$. So for $j\geq 2$, $\supp (f_{1,j}) \subseteq 
f^{-1}_1([j\note{\epsilon},\infty))$ and $\supp f_{2,M+1-j}\subseteq f^{-1}_2([(M+1-j)\note{\epsilon},\infty))$. 

If $x\in \supp (f_{1,j}) \cap \supp (f_{2,M+1-j})$ then 
$1\note{\geq} f_1(x) + f_2(x) \geq j\eps +(M+1-j)\eps = 1+\eps$, so the two supports  $\supp (f_{1,j}), \supp (f_{2,M+1-j})$ are disjoint. 

We have that 
\[
\begin{split}
f_1 & = f_{1,0} + f_{1,1} + \sum_{j=2}^{M-1} f_{1,j}\\
f_2& = f_{2,0} + f_{2,1} + \sum_{j=2}^{M-1}f_{2,M+1-j}. 
\end{split}
\]
So using the fact that $\norm{
f_{1,j}\phi_1+f_{2,M+1-j}\phi_2} \leq \sup \{
\note{\norm{f_{1,j}\phi_1}},\note{\norm{f_{2,M+1-j}\phi_2}}\} \leq \eps\sup_i\norm{\phi_i}$ we have the following estimate: 
\[
\begin{split}
\norm{f_1\phi_1 + f_2\phi_2} &\leq \norm{(f_{1,0} + f_{1,1})\phi_1} + 
\norm{(f_{2,0}+f_{2,1})\phi_2} + \sum_{j=2}^M\norm{
f_{1,j}\phi_1+f_{2,M+1-j}\phi_2} \\
& \leq 4\eps \sup_j\norm{\phi_i} + \sum_{j=2}^{M-1}\eps\sup_{\note{i}}\note{\norm{{\phi_i}}}\\
& = (4\eps +(M-2)\eps)\sup_i\norm{\phi_i}\\
& = (1+2\eps)\sup_i\norm{\phi_i}.
\end{split}
\]

\end{proof}

\begin{lemma}
Let $f_1,\dots, f_N\in \cx$, $f_i\geq 0$, $\sum_{i=1}^Nf_i \leq 1_X$, $\phi_1,\dots,\phi_N\in \cale^*$.

Then $\norm{\sum_i f_i\phi_i} \leq \sup_{1,\dots,N}\norm{\phi_i}$. 
\end{lemma}
\begin{proof}
We proceed by induction. Assume that the statement is true for some $N$. Then let $f_0,f_1,\dots, f_N\in \cx$, $f_i\geq 0$, $\sum_{i=1}^N
\note{f_i}\leq 1_X$, \note{and let} $\phi_0,\phi_1,\dots,\phi_N\in \cale^*$.

Let $f'_1= f_0+f_1$ and leave the other functions unchanged. \note{For $\delta >0$ let}
\[
\phi'_{1,\delta} = \frac{1}{f_0+f_1+\delta}(f_0\phi_0+f_1\phi_1).
\]
Since we clearly have 
\[
 \frac{f_0}{f_0+f_1+\delta} + \frac{f_1}{f_0+f_1+\delta}\leq \note{1_X}
 \]
 by the previous lemma we have that 
 $\norm{\phi'_{1,\delta}} \leq \sup\note{\{}\norm{\phi_0},\norm{\phi_1}\note{\}}$, and so \note{by induction}
 \[
 \norm{f'_1\phi'_{\note{1,\delta}} +f_2\phi_2+\dots +f_N\phi_N} \leq \note{\sup \{\norm{\phi'_{1,\delta}}, \norm{\phi_2}, \dots, \norm{\phi_N}\}} 
 \leq \sup_{i=0,\dots, N}\norm{\phi_i}. 
 \]
 
 Consider now 
 \[
 f'_1\phi'_{1,\note{\delta}} =   \frac{\note{(f_0+f_1)}}{f_0+f_1+\delta}(f_0\phi_0+f_1\phi_1) = 
 \note{\frac{(f_0+f_1)f_{0}}{f_0+f_1+\delta}\phi_0 + 
 \frac{(f_0+f_1)f_{1}}{f_0+f_1+\delta}\phi_1}.
 \]
 
 We note that 
 \note{for $i=0,1$} 
 \[
 f_{\note{i}} -  \frac{(f_0+f_1)f_{\note{i}}}{f_0+f_1+\delta} =  \frac{\delta f_{\note{i}}}{f_0+f_1+\delta} \leq \delta
 \]
 and so $\frac{(f_0+f_1)f_{\note{i}}}{f_0+f_1+\delta}$ converges to $f_{\note{i}}$ uniformly on $X$, as $\delta\to 0$, which implies that
 $f'_1\phi'_{1,\delta}$ converges to $f_0\phi_0+f_1\phi_1$ in norm, and the lemma follows. 
\end{proof}

\begin{lemma}\label{estimate}
If $f_1,\dots, f_N\in \cx$ (we do not assume that $f_i\geq 0$) are such that $\sum_{i=1}^N |f_i| \leq 1_X$ and 
$\phi_1,\dots, \phi_N\in \cale^*$ then
\[
\norm{\sum_{i=1}^Nf_i\phi_i}\leq 2\sup_{i=1,\dots, N}\norm{\phi_i}.
\]
\end{lemma}
\begin{proof}
If $f_i = f_i^+ - f_i^-$, then $|f_i| = f_i^+ +f^-_i$ and $\sum f_i^+ + \sum f^-_i \leq 1$. 

Then by the previous lemma $\norm{\sum_{i=1}^Nf^{\pm}_i\phi_i }\leq \sup_{i=1,\dots, N}\norm{\phi_i}$
so 
\[
\norm{\sum f^+_i\phi_i- \sum f_i^-\phi_i}\leq 2\sup_{\note{i=1,\dots, N}} \norm{\phi_i}.
\]

\end{proof}

\section{Amenable actions and invariant means}
In this section we will recall the definition of a topologically amenable action
and characterise it in terms of the existence of a certain averaging operator. For our purposes the following definition, adapted from \cite[Definition 4.3.1]{BrownOzawa} is convenient.


\begin{defn}\label{action'}
The action of $G$ on $X$ is amenable if and only if there exists a sequence of 
elements $f^n\in W_{00}(G,X)$ such that  
\begin{enumerate}
\item $f^n_g\ge0$ in $C(X)$ for every $n\in \N$ and $g\in G$, 
\item $\pi(f^n)=1$ for every $n$,
\item for each $g\in G$ we have $\Vert f^n-g f^n\Vert_V\to 0$.
\end{enumerate}
\end{defn}

Note that when $X$ is a point the above conditions reduce to the definition of amenability
of $G$. On the other hand, if $X=\beta G$, the Stone-\v{C}ech compactification of $G$
then amenability of the natural action of $G$ on $X$ is equivalent to Yu's property A
by a result of Higson and Roe \cite{HR}.

\begin{remark}\label{normalise} In the above definition we may omit condition 1 at no cost, since given a sequence of functions satisfying conditions $2$ and $3$ we can make them positive by replacing each $f^n_g(x)$ by

\[
\frac{|f^n_g(x)|}{\sum\limits_{h\in G}|f^n_h(x)|}.
\]

Conditions $1$ and $2$ are now clear, while condition $3$ follows from standard estimates (see e.g. \cite[Lemma 4.9]{DouglasNowak}).

%
\end{remark}

The first definition of amenability of a group $G$ given by von Neumann was in terms 
of the existence of an invariant mean on the group. The following definition  gives
a version of an invariant mean for an amenable action on a compact space.

\begin{defn}\label{mean}
Let $G$ be a countable group acting on a compact space $X$ by homeomorphisms.  
A mean for the action is an element $\mu\in \wo(G,X)^{**}$ such that $\mu(\pi)=1$.
A mean $\mu$ is said to be invariant if $\mu(g \varphi)=\mu(\varphi)$
for  every $\varphi\in W_{0}(G,X)^*$.
\end{defn}


We now state our first main result.

\begin{mainthm}
Let $G$ be a countable discrete group acting by homeomorphisms on a compact Hausdorff topological space $X$. The action is amenable
if and only if there exists an invariant mean for the action.
\end{mainthm}

\begin{proof}
Let $G$ act amenably on $X$ and 
consider the sequence $f^n$ provided by Definition \ref{action'}.
Each $f^n$ satisfies $\Vert f^n\Vert=1$.
We now view the functions $f^n$ as elements of the double dual $W_0(G,X)^{**}$.
By the weak-* compactness of the unit ball there is a convergent subnet $f^{\lambda}$,
and we define $\mu$ to be its weak-* limit. It is then easy to verify that $\mu$ is a mean.
Since 
$$\vert \langle f^{\lambda}-gf^{\lambda}, \varphi\rangle\vert \le \Vert f^{\lambda}-gf^{\lambda}\Vert_V \Vert \varphi\Vert$$
and the right  hand side tends to 0, we obtain $\mu(\varphi)=\mu(g\varphi)$.

Conversely, by Goldstine's theorem, (see, e.g., \cite[Theorem 2.6.26]{Meg}) as $\mu \in \wo(G,X)^{**}$, $\mu$ is the weak-* limit of a bounded net of elements $f^\lambda \in \wo(G,X)$. 
We note that we can choose $f^\lambda$ in such a way that $\pi(f^{\lambda})=1$. Indeed, given $f^{\lambda}$
with $\pi(f^\lambda)=c_\lambda\to \mu(\pi)=1$ we replace each $f^\lambda$ by  
$$f^\lambda+(1-c_\lambda)\delta_e.$$
Since  $(1-c_\lambda)\delta_e\to 0$ in norm in $\wo(G,X)$,  $\mu$  is the weak-* limit of the net 
$f^\lambda+(1-c_\lambda)\delta_e$ as required.

Since $\mu$ is invariant, 
we have that for every $g\in G$, $g {f}^\lambda \to g\mu= \mu$, so that $g {f}^\lambda - {f}^\lambda \to 0$ in the weak-* topology. 
However, for every $g\in G$, $gf^\lambda - f^\lambda \in \wo(G,X)$, and so the convergence is in fact in the weak topology on 
$\wo(G,X)$.

For every $\lambda$, we regard the family $(gf^\lambda-f^\lambda)_{g\in G}$ as an element of the product 
$\prod_{g\in G} \wo(G,X)$, noting that this sequence converges to $0$ in 
the Tychonoff weak
  topology.


Now $\prod_{g\in G} \wo(G,X)$ is a Fr\'{e}chet space in the Tychonoff norm topology, so by Mazur's theorem there exists a sequence $f^n$ of convex combinations 
of $f^\lambda$ such that $(gf^n- f^n)_{g\in G}$ converges to zero in the Fr\'{e}chet topology. Thus there exists a sequence $f^n$ of elements of $\wo(G,X)$ such that for every $g\in G$, 
$\Vert gf^n - f^n\Vert \to 0$  in  $\wo(G,X)$.


The result then follows from 
Remark \ref{normalise}.
\end{proof}

\section{Equivariant means on geometric modules}

Given an invariant mean $\mu\in \wo(G,X)^{**}$ for the action of $G$ on $X$ and an $\ell^1$-geometric $\gcx$ module $\cale$, we define a $G$-equivariant averaging operator $\mu_\cale: \ell^\infty(G,\cale^*)\ra \cale^*$ which we will also refer to as an equivariant mean for the action.

To do so, following an idea from \cite{BrodzkiNibloWright}, we introduce a linear functional $\sigma_{\tau,v}$ on $\woo(G,X)$. 
Given a Banach space $\cale$ define $\ell^\infty(G,\mathcal{E})$ to be the space of functions 
$f:G\to\cale$ such that $\sup_{g\in G}\Vert f(g)\Vert_\cale<\infty$.
If $G$ acts on $\cale$ then the action of the group $G$ on the space $\ell^\infty(G,\cale)$ is 
defined in an analogous way to the action of $G$ on $V$, using the induced action of $G$ on 
$\cale$:
$$(g\bullet \tau)_h=g(\tau_{g^{-1}h}),$$
for $\tau\in \ell^\infty(G,\cale)$ and $g\in G$.

Let us assume that $\cale$ is an $\ell^1$-geometric $\gcx$ module, and let $\tau\in \ell^\infty(G,\cale^*)$. Choose a vector $v\in \cale$ and define a linear functional 
  $\sigma_{\tau,v}: \woo(G,X)\ra \R$ by 
 \begin{equation}\label{sigma}
 \sigma_{\tau,v} (f) = \langle\sum_{h\in G}f_h\tau_h,v\rangle
 \end{equation}
 for every $f\in \woo(G,X)$. 
If we now use  Lemma \ref{estimate} together with the support condition required of elements of $\woo(G,X)$ then we have 
the estimate
\[
|\sigma_{\tau,v}(f)|\leq \Big\Vert \sum_h f_h\tau_h\Big\Vert \norm{v} \leq 2\norm{f}\norm{\tau}\norm{v}.
\]
 This estimate completes the proof of the following. 
 \begin{lemma}
Let $\cale$ be an $\ell^1$-geometric $\gcx$ module. For every $\tau\in \ell^\infty(G,\cale^*)$ and every $v\in \cale$ the linear functional $\sigma_{\tau,v}$ on  $\woo(G,X)$ is continuous and so it extends 
to a continuous linear functional on $\wo(G,X)$. 
\end{lemma}

 

 \begin{lemma}
 The map $\ell^\infty(G,\cale^*)\times \cale\ra \wo(G,X)^*$ defined by $(\tau,v) \mapsto \sigma_{\tau,v}$ is $G$-equivariant. 
 \end{lemma}
 \begin{proof}
 \[
 \begin{split}
 \sigma_{g\bullet\tau,gv}(f) & = \left\langle\sum_hf_h g(\tau_{g^{-1}h}),gv \right\rangle = 
 \left \langle g \sum_h(g^{-1}\cdot f_h)\tau_{g^{-1}h} ,gv \right\rangle \\
 & =   \left\langle \sum_h(g^{-1}\cdot f_h)\tau_{g^{-1}h},v \right\rangle = 
  \left \langle \sum_h(g^{-1}f)_{g^{-1}h}\tau_{g^{-1}h},v \right\rangle \\
  & = \sigma_{\tau,v}(g^{-1} f) = (g \sigma_{\tau,v})(f). 
  \end{split}
  \]
 
 \end{proof}
 \begin{defn}
Let $\cale$ be an $\ell^1$-geometric $\gcx$ module, and let $\mu\in \wo(G,X)^{**}$ be an invariant mean for the action. We define $\mu_\cale: \ell^\infty(G,\cale^*)\ra \cale^*$ by 
\[
 \langle\mu_\cale(\tau),v\rangle = \langle\mu,\sigma_{\tau,v}\rangle, 
\]
for every $\tau\in \ell^\infty(G, \cale^*)$, and $v\in \cale$. 
\end{defn}

\begin{lemma} \label{mu} Let $\cale$ be an $\ell^1$-geometric $\gcx$ module, and let $\mu\in \wo(G,X)^{**}$ be an invariant mean for the action. 
\begin{enumerate}
\item The map $\mu_\cale$ defined above is $G$-equivariant. 
\item If $\tau \in \ell^\infty(G, \cale^*)$ is constant then $\mu_\cale(\tau) = \tau_e$. 
\end{enumerate}
\end{lemma}
\begin{proof}
\[
\begin{split}
\langle\mu_\cale(g\bullet \tau),v\rangle & = \mu(\sigma_{g\bullet\tau,v}) = \mu(g\cdot \sigma_{\tau,g^{-1}v}) = \mu(\sigma_{\tau,g^{-1}v})\\
& = \langle\mu_\cale(\tau),g^{-1}v\rangle = \langle g\cdot (\mu_\cale(\tau)),v\rangle. 
\end{split}
\]

If $\tau$ is constant then 
\[
\begin{split}
\sigma_{\tau,v}(f)& = \left\langle\sum_h f_h\tau_h, v\right\rangle = \left\langle\left(\sum_h f_h\right)\tau_e,v\right\rangle
\\
& = \langle(\pi(f)1_X)\tau_e,v\rangle = \langle\pi(f)\tau_e,v\rangle = \langle\tau_e,v\rangle\pi(f) . 
\end{split}
\] 
So $\sigma_{\tau,v} = \langle\tau_e,v\rangle\pi$ and 
\[
\langle\mu_\cale(\tau),v\rangle = \mu(\sigma_{\tau,v})= \mu(\langle\tau_e,v\rangle\pi) = \langle\tau_e,v\rangle, 
\]
hence $\mu_\cale(\tau) = \tau_e$. 
\end{proof}

\section{Amenable actions and bounded cohomology}

Let 
$\cale$ be a Banach space equipped with an isometric action by $G$. 
Then we consider a cochain complex $C^{m}_b(G,\cale^*)$ which in degree $m$ consists of $G$-equivariant bounded cochains 
$\phi:G^{m+1}\ra \cale^*$ with values in the Banach dual $\cale^*$ of $\cale$ which is equipped with the natural differential $d$ as in the homogeneous bar resolution.  Bounded cohomology with coefficients 
in $\cale^*$ will be denoted by $H^*_b(G,\cale^*)$.  

\begin{defn}
Let $G$ be a countable discrete group acting by homeomorphisms on a compact Hausdorff topological space $X$. 
The function 
\[
J(g_0,g_1) = \delta_{g_1} - \delta_{g_0}
\]
is a bounded cochain of degree $1$ with values in $\noo(G,X)$, and in fact it is a bounded cocycle and so represents a class 
in $H^1_b(G,\no(G,X)^{**}$, where we regard $\noo(G,X)$ as a subspace of $\no(G,X)^{**}$. 
We call $[J]$ the Johnson class of the action.
\end{defn}

\begin{mainthmtwo}
Let $G$ be a countable discrete group acting by homeomorphisms on a compact Hausdorff topological space $X$. Then the following are equivalent
\begin{enumerate}
\item \label{amen} The action of $G$ on $X$ is topologically amenable.
\item \label{J} The class $[J]\in H_b^1(G, N_0(G,X)^{**})$ is trivial.
\item \label{vanish}   $H_b^p(G,\cale^*) = 0 $ for $p\geq 1$ and every $\ell^1$-geometric $\gcx$ module $\cale$. 
\end{enumerate}
\end{mainthmtwo}

 
\begin{proof}
We first show that (\ref{amen}) is equivalent to (\ref{J}). The short exact sequence of $G$-modules
\[
0\ra \no(G,X) \ra \wo(G,X) \xrightarrow{\pi} \R\ra 0
\]
leads, by taking double duals, to the short exact sequence
\[
0\ra \no(G,X)^{**}\ra \wo(G,X)^{**} \ra \R\ra 0
\]
which in turn gives rise to a long exact sequence in bounded cohomology
\[
H^0_b(G,\no(G,X)^{**}) \ra H^0_b(G,\wo(G,X)^{**}) \ra H^0_b(G,\R) \ra H^1_b(G,\no(G,X)^{**})\ra \dots
\]
The Johnson class $[J]$ is the image of the class $[1] \in H^0_b(G, \R)$ under the connecting homomorphism  $d: H^0_b(G,\R) \ra H^1(G,\no(G,X)^{**})$, and so $[J]=0$ if and only if $d[1] = 0$. By exactness of the cohomology sequence, this is equivalent to  $[1]\in \text {Im}\;\pi^{**}$, where 
$\pi^{**}: H^0_b(G,\wo(G,X)^{**}) \ra H^0_b(G,\R)$ is the map on cohomology induced by the summation map $\pi$.  Since $H^0_b(G,\wo(G,X)^{**})=\left(\wo(G,X)^{**}\right)^G$ and
$H_b^0(G,\R)=\R$ we have that $[J]=0$ if and only if
there exists an element $\mu\in \wo(G,X)^{**}$ such that $\mu=g\mu$ and $\mu(\pi)=1$.
Thus $\mu$ is an invariant mean for the action 
and the equivalence with amenability of the action 
follows from Theorem  A.

We turn to the implication (\ref{amen}) implies (\ref{vanish}). Since $G$ acts amenably on $X$ there is, by Theorem A,  an invariant mean $\mu$ associated with the action.
%
For every $h\in G$ and for every  equivariant bounded cochain $\phi$ we define
$s_h\phi: G^{p}\rightarrow \cale^*$ by
$s_h\phi(g_0, \dots, g_{p-1})= \phi(g,g_0, \dots, g_{p-1})$; we note that for fixed $h$,  $s_h\phi$ is not equivariant in general. However, the map $s_h$ does satisfy the identity $ds_h + s_hd = 1$ for every $h\in G$, and we will now
construct an equivariant contracting homotopy, adapting an averaging procedure introduced in \cite{BrodzkiNibloWright}. 

For $\phi\in C^p_b(G,\cale^*)$ let $\widehat{\phi}:G^p\rightarrow \ell^\infty(G, \cale^*)$ be defined by $\widehat{\phi}({\bf g})(h)=s_h\phi({\bf g})$, for ${\bf g}=(g_0, \ldots g_{p-1})$.



Note that  for every $k,h\in G$, 
\[
\begin{split}
{\widehat{\phi}(kg_0,\dots, kg_{p-1})}(h) & = \phi(h, kg_0,\dots, kg_{p-1}) = k (\phi(k^{-1}h, g_0, \dots, g_{p-1})) \\
&= k (\widehat{\phi}(g_0, \dots, g_{p-1})(k^{-1}h))\\
&= (k (\widehat{\phi}(g_0, \dots, g_{p-1})))(h)\\
\end{split}
\]
so  $\widehat{\phi}(k{\bf g})=k(\widehat{\phi}({\bf g}))$.

We can now define a map
$s:  C^p(G,\cale^*)\ra C^{p-1}(G,\cale^*)$:
\[
s\phi({\bf g}) = \mu_\cale(\widehat{\phi}({\bf g})), 
\]
where $\mu_\cale: \ell^\infty(G,\cale^*) \ra \cale^*$ is the map defined in Lemma \ref{mu} using the invariant mean  $\mu$. Note that $\|\mu_\cale\| \leq 2 \|\mu\|$, and $\|\widehat{\phi}({\bf g})\|\leq \sup\{\|\phi({\bf k})\|\mid {{\bf k}\in G^{p+1}}\}$. Hence $s\phi$ is bounded.

For every cochain $\phi$, $k(s\phi)=s(k\phi)=s\phi$ since $\widehat{\phi}$ and $\mu_\cale$ are equivariant.

The map $s$ provides a contracting homotopy for the complex $C^*_b(G,\cale^*)$ which can be seen as follows. 
As $\mu_\cale: \ell^\infty(G,\cale^*) \ra \cale^*$ is a linear operator it follows that for a given  $\phi\in C^p_b(G,\cale^*)$, and a 
$p+1$-tuple of arguments $\mathbf k = (k_0,\dots, k_{p})$, $ds\phi$ is obtained by applying the mean
$\mu_\cale$ to the map $g\mapsto ds_g\phi(\mathbf k)$, while $sd\phi$ is obtained by applying $\mu_\cale$ to the function 
$g\mapsto s_gd\phi (\mathbf k)$. Thus
\[
(sd + ds)\phi(\mathbf k) = \mu_\cale(g\mapsto (ds_g +s_gd)\phi(\mathbf k)). 
\]
Given that $ds_g +s_gd = 1$ for every $g\in G$, for every ${\bf g}\in G^{p+1}$  the function $g\mapsto (ds_g +s_gd)\phi(\mathbf k) = \phi(\mathbf k) \in \cale^*$ is constant, 
and so by Lemma \ref{mu}, 
\[
(sd + ds)\phi(\mathbf k) = (ds_e +s_ed)\phi(\mathbf k) = \phi(\mathbf k).
\]
Thus $sd + ds = 1$, as required. 


Collecting these results together, we have proved that (\ref{amen}) implies (\ref{vanish}). 


The fact that  (\ref{vanish}) implies (\ref{J}), follows from the fact that $\no(G,X)^{*}$ is an $\ell^1$-geometric \gcx-module, proved in Lemma \ref{lemma : N0 is ell-infty geometric}. 

\end{proof}

\end{document}